\documentclass[reqno,12pt, letter]{amsart}

\title[Bubbles and the two-point Schramm observable]{Some remarks on SLE bubbles and Schramm's two-point Observable }

\author[D. Beliaev]{Dmitry Beliaev}
\address{Dmitry Beliaev\\ 
Department of Mathematics \\
Princeton University\\
Princeton, USA}
\thanks{}
\email{dbeliaev@math.princeton.edu}
\author[F. Johansson Viklund]{Fredrik Johansson Viklund}
\address{Fredrik Johansson Viklund\\
Department of Mathematics\\
Columbia University\\
New York, USA}

\thanks{}
\email{fjv@math.columbia.edu}
\usepackage{fullpage}
\usepackage{enumerate}
\usepackage{verbatim}
\usepackage{amssymb,amsmath,amsthm}
\usepackage{graphicx}
\usepackage{epsfig}

\newtheorem*{theorem*}{Theorem}
\newtheorem{Theorem}{Theorem}[section]
\newtheorem{corollary}[Theorem]{Corollary}

\newtheorem{proposition}[Theorem]{Proposition}

\theoremstyle{remark}
\newtheorem*{remark}{Remark}

\newcommand{\ball}{\mathcal{B}}

\newcommand{\HH}{{\mathbb H}}

\renewcommand{\Im}{{\rm Im}}

\newcommand{\ee}{\epsilon}

\let \le \leqslant

\let \ge \geqslant

\numberwithin{equation}{section}

\begin{document}

\begin{abstract}
Simmons and Cardy recently predicted a formula for the probability that the chordal SLE$_{8/3}$ path passes to the left of two points in the upper half-plane. In this paper we give a rigorous proof of their formula. Starting from this result, we derive explicit expressions for several natural connectivity functions for SLE$_{8/3}$ bubbles conditioned to be of macroscopic size. By passing to a limit with such a bubble we construct a certain chordal restriction measure and in this way obtain a proof of a formula for the probability that two given points are between two commuting SLE$_{8/3}$ paths. The one-point version of this result has been predicted by Gamsa and Cardy. Finally, we derive an integral formula for the second moment of the area of an SLE$_{8/3}$ bubble conditioned to have radius $1$. We evaluate the area integral numerically and relate its value to a hypothesis that the area follows the Airy distribution.  \end{abstract}

\maketitle
\section{Introduction and Results}
A range of planar (critical) lattice models from statistical physics have scaling limits that can be described using the Schramm-Loewner evolution (SLE$_\kappa$), a family of conformally invariant random fractal curves. Examples include loop-erased random walks, uniform spanning trees, critical percolation, the Ising model, Gaussian free fields, and, conjecturally, self-avoiding walks. There are also natural variants of SLE that describe random loops arising in these models. Besides being an object of independent interest, the SLE process has provided techniques for a rigorous approach to many aspects of these models. These include---but are certainly not limited to---the passage to the scaling limit itself as well as a method to derive and treat rigorously some of the differential equations that correlation functions in the continuum limit conformal field theory (CFT) satisfy. 

One of the simplest SLE observables is the probability that the curve, or interface, passes to the left of a given point in the upper half plane. Schramm found an explicit expression for this observable and used it to derive a new connectivity function for (the scaling limit of) critical percolation interfaces, see \cite{Sch01}. For the case
$\kappa = 8/3$, which is believed to be the scaling limit of the self-avoiding walk\footnote{More precisely, if the scaling limit of self-avoiding walk is conformally invariant, then it is an SLE$_{8/3}$ curve, see \cite{LSW_SAW} for more details}, Simmons and Cardy predicted an analogous formula
for two points using CFT techniques combined with some intricate complex coordinate changes, see
\cite{SiCa}. One purpose of this note is to give a rigorous proof of
this latter identity and, perhaps more importantly, to bring it to attention to the mathematical community; this is one of very few cases when a two-point correlation for SLE is known explicitly. 

We also give applications of the one- and two-point functions, mainly to variants of SLE$_{8/3}$ {\bf bubbles}. Here, a bubble is a Jordan loop in the upper half-plane ``attached'' to one point on the real line. Such random bubbles are rather natural objects which are known to be closely related to Brownian bubbles, Brownian excursions, and conjecturally to scaling limits of so-called self-avoiding (boundary) polygons, see \cite{LSW_restriction} and \cite{LSW_SAW}. Similar SLE$_{\kappa}$ bubbles with $\kappa \in (8/3,4)$, called pinned loops, were recently studied and used by Sheffield and Werner in the context of conformal loop ensembles. We consider several different probability measures on (hulls of) SLE$_{8/3}$ bubbles of macroscopic size and the explicit one- and two-point functions together with the special restriction property allow us to compute explicit quantities for these objects, see Proposition~\ref{formulas} and its corollaries. These computations then lead to further results: 
%contain a given point $z$, the boundary of which is the limit as $\ee \to 0$ of an SLE$_{8/3}$ curve from $0$ to $\ee$ %conditioned to pass to the left of $z$. 
%(The filling is the closure of the component of the complement of the bubble's boundary not bounded by part of the real line.) 

The area distribution of the hull of a self-avoiding polygon (and related random loops) has been studied by a number of authors in the mathematical physics literature. Numerical investigation and theoretical considerations have led to predictions that it properly normalized asymptotically follows the Airy distribution, see, e.g., \cite{CR1} and the references therein. (This distribution appears in other related contexts, too; the area under the graph of a one-dimensional Brownian excursion of fixed time length being one of the more well-known examples.) Given the conjectured relationship with self-avoiding boundary polygons, one may wonder whether the Airy distribution can be used to describe the area of the hull of a suitably normalized SLE$_{8/3}$ bubble. Motivated by this and the work by Garban and Trujillo Ferreras in \cite{GTF}, who computed the first moment, we use the two-point function to derive an explicit, albeit complicated, formula for the second moment of the area of the hull of an SLE$_{8/3}$ bubble conditioned to have radius $1$. The hypothesis that this area is Airy distributed combined with the exact value of the first moment implies an exact value of the second moment, providing a way to test the hypothesis. However, the formula for the second moment involves a multiple integral of hypergeometric functions and seems difficult to evaluate exactly. Numerical evaluation supports to some extent the above hypothesis.  

In the last section we consider infinite SLE bubbles. We show that the law of the hull of an SLE$_{8/3}$ bubble conditioned to have large radius  (or to contain a point with large modulus) approaches the chordal restriction measure with exponent $2$, see Proposition~\ref{apr13.1}. Samples from this probability measure can be interpreted in terms of two commuting (mutually avoiding) SLE$_{8/3}$ curves started from the same point. This effectively converts Schramm's observables for one curve to analogous observables for the system of two curves. In particular, we can compute explicitly the probability that two given points are between two commuting SLE$_{8/3}$ paths.
% equals
%\[
%-\frac{2}{5}\left(\Im \, z \, \Im \,  w^{-1} + \Im \, z^{-1} \, \Im \, w \right)G(\sigma),\]
%where $G$ is given explicitly in terms of a hypergeometric function and $\sigma=|z-w|^2/|z-\overline{w}|^2$. 
When the points collapse to one, the expression becomes $(4/5)\sin^2(\arg(z))$ and proves a prediction by Gamsa and Cardy \cite{GC}, see Proposition~\ref{gamsa-cardy}. Gamsa and Cardy found this one-point function by solving a third-order differential equation which was derived by CFT methods. We may note that this equation is not of the type that easily translates to SLE language; in particular, it does not come from a so-called level two degeneracy.  
   
%The connection to Brownian loops was exploited by Garban and Trujillo Ferreras in \cite{GTF} who computed the expected area of the filling of an SLE$_{8/3}$ bubble conditioned to touch the unit circle. (This number turns out to equal half the expected oaf of the filling of a Brownian (whole-plane) excursion run up to time $1$ or, equivalently, the filling of the outer boundary of a Brownian loop.)  

To keep the present article short we will not provide a detailed preliminary discussion but instead refer the reader to \cite{RoSch}, \cite{LSW_restriction}, and \cite{GTF} (and the references therein) for further background information. Let us also make the following comment. For convenience we will sometimes casually refer to probability laws conditioned on zero-probability events. (Indeed, we have already.) All such distributions will be defined formally in the text---usually by an appropriate sequence of limits---even though we will not always reference the definition. 
\subsection{Acknowledgements}
Beliaev is partially supported by NSF grant DMS-0758492 and the Chebyshev Laboratory (Faculty of Mathematics and Mechanics, St Petersburg State University) under the grant 11.G34.31.0026 of the Government of the Russian Federation. Johansson Viklund acknowledges support from the Simons Foundation and Institut Mittag-Leffler.
\section{Two-Point Function}

Write $z=x+iy$ and $w=u+iv$ for two points in the upper half-plane $\mathbb{H}$. Let $(g_t)$ be the chordal (forward) SLE$_\kappa$ Loewner chain, that is,
\begin{equation}\label{LE}
\partial_t g_t(z) = \frac{2}{g_t(z) - \sqrt{\kappa}B_t}, \quad 0 < t < \tau(z), \quad g_0(z)=z,
\end{equation}
where $B$ is a standard Brownian motion and $\tau(z)$ is the blow-up time for \eqref{LE}. Let
$z_t:=g_t(z)-\sqrt{\kappa}B_t$, $w_t:=g_t(w)-\sqrt{\kappa}B_t$, and
write $x_t+iy_t=z_t$ and $u_t+iv_t=w_t$. Let $\gamma$ be the standard chordal SLE$_\kappa$ curve in $\mathbb{H}$ so that for each $t \ge 0$
\[
\gamma(t) = \lim_{y \to 0+} g_t^{-1}(iy + \sqrt{\kappa}B_t).
\]
It is a non-trivial fact that $t \mapsto \gamma(t)$ is almost surely a continuous curve (in $\mathbb{H}$ growing from $0$ to $\infty$) which is simple if and only if $0 \le \kappa \le 4$, see \cite{RoSch}.

Define $L_{\kappa}(x,y,u,v)$
to be the probability that $\gamma$ passes to the left of both $z$ and
$w$. More precisely, this should be understood as the harmonic measures
at $z$ and $w$ of the right hand side of the curve in union with the
positive real line both tend to one as $t \to \infty$. (This is the case at least when $\kappa \le 4$. Otherwise we may consider the same type of event as $t \to \max\{\tau(z), \tau(w)\}$.) It is easy to see from the expression for the harmonic measure of the positive real line in $\mathbb{H}$, and proved in Lemma~3 of \cite{Sch01}, that this event is equivalent to $x_t/y_t$ and $u_t/v_t$ both tending to $+\infty$ as $t \to \infty$. We have the following result, which was predicted in \cite{SiCa}.
\begin{Theorem}\label{mainthm}
Let $\gamma$ be the standard chordal SLE$_{8/3}$ path and suppose $z=x+iy, \,w=u+iv \in \mathbb{H}$. Set $\sigma=|z-w|^2/|z-\overline{w}|^2$. Then
\begin{multline}
\mathbb{P} \left\{ \gamma \, \textrm{ passes to the left of both } z \textrm{ and } w \right\}
\\= \left(\frac{1}{2}+\frac{x}{2|z|}\right)\left(\frac{1}{2}+\frac{u}{2|w|}\right)\left(1+\frac{y}{x+|z|}\frac{v}{u+|w|}G(\sigma) \right),
\end{multline}
where 
\begin{equation}\label{G}
G(\sigma)=1-\sigma {}_2 F_1(1, 4/3;5/3; 1-\sigma)
\end{equation}
and ${}_2 F_1$ is the hypergeometric function.
\end{Theorem}
\begin{remark}
The one-point function for $\kappa=8/3$, with the same notation as in the theorem, is given by
\[
\mathbb{P} \left\{ \gamma \, \textrm{ passes to the left of } z \right\}
= \frac{1}{2}+\frac{x}{2|z|}.
\]
\end{remark}
\begin{remark}
In a similar manner as above one can define probabilities of the remaining three non-trivial outcomes, giving three additional two-point functions. The formulas for these probabilities could be
proved in exactly the same way as Theorem~\ref{mainthm}, or could be derived from it using Schramm's
one-point formula as given in the previous remark. 
\end{remark}
\begin{proof}[Proof of Theorem~\ref{mainthm}]
Set $\kappa = 8/3$ and write $L=L_{8/3}$.
By the conformal Markov property of SLE we have
\[
\mathbb{E}[L(x,y,u,v) | \mathcal{F}_t] = L(x_t,y_t,u_t,v_t)
\]
so that $t \mapsto L(x_t,y_t,u_t,v_t)$ is an invariant SLE martingale. ($\mathcal{F}_t$ is the filtration generated by the driving Brownian motion.)
Hence if
\[
\Lambda=\frac{2x}{x^2+y^2} \partial_x+\frac{-2y}{x^2+y^2} \partial_y +\frac{2u}{u^2+v^2} \partial_u+\frac{-2v}{u^2+v^2} \partial_v 
+ \frac{\kappa}{2} \left( \partial_{x} + \partial_u \right)^2
\]
then, assuming for the moment that $L \in \mathcal{C}^2$, we see from \eqref{LE} using It\^o's formula that the martingale property implies that
\begin{equation}
\label{martingale}
\Lambda L= 0.
\end{equation}
Following Simmons and Cardy, we consider now the auxiliary variable $\sigma=\sigma(z,w)$ defined by
\[
\sigma=\frac{|z-w|^2}{|z-\overline{w}|^2}=\frac{(x-u)^2+(y-v)^2}{(x-u)^2+(y+v)^2},
\]
which is the exponential of (a multiple of) the usual Green function for the half-plane, 
the basic conformal invariant for configurations with two interior points.
We have that $\left( \partial_{x} + \partial_u \right)^2H(\sigma)=0$ for
any $\mathcal{C}^2$ function $H$ and that $\sigma=0$ corresponds to $z=w$, that is,
the fully correlated case. The value $\sigma =1$ corresponds to at
least one point being on the real line or one point being sent to infinity. In these cases we are back to Schramm's original problem. Let $L(x,y)=L(x,y,x,y)$ be the one-point left-passage probability. We search for a solution of \eqref{martingale} in the form
\begin{equation}\label{apr29.1}
L_{}(x,y)L_{}(u,v)\left(1+\frac{\sqrt{1-L_{}(x,y)}\sqrt{1-L_{}(u,v)}}{\sqrt{L_{}(x,y)}\sqrt{L_{}(u,v)}}G(\sigma)\right),
\end{equation}
where we require that 
\[
\lim_{\sigma \to 0+}G(\sigma)=1, \quad \lim_{\sigma \to 1-}G(\sigma)=0.
\]
This ansatz (which should be credited to Simmons and Cardy) is quite natural. Indeed, the two-point function
is a conformal invariant for configurations with two boundary points and two interior points. Moreover, we expect it to look similar to the product of two one-point functions,
to be symmetric with respect to the interior points, to reduce
to a one-point function when the interior points collapse to one, and to reduce to the product of the one-point functions when one point is on the real line (or is sent to infinity). We also expect the renormalized two-point function for the probability that the path passes between two points to behave as the SLE Green function when the points are collapsed to one, see the remark below. The above ansatz is arguably the simplest expression satisfying these conditions together with the requirements on conformal invariance. (Of course, in general, there is no reason to expect a conformal invariant for a four-point configuration as above to ``factorize'' into simpler conformal invariants.) 
%However, certain probabilities related to (continuum) percolation evens are known to factorize in a similar way, see %\cite{BeIz}.

From \cite{Sch01} we have
\[
L(x,y)=\frac{1}{2}\left(1+\frac{x}{\sqrt{x^2+y^2}}\right).
\]

Plugging this into \eqref{apr29.1} gives an ansatz for $L(x,y,u,v)$. Using the identity 
\[\frac{4yv}{(x-u)^2+(y+v)^2}=1-\sigma\]
the equation $\Lambda L(x,y,u,v)=0$ now implies that $G$ must satisfy a hypergeometric ODE, namely
\[
t-1 + (t+1) G(t) - 3 t (1-t) G'(t)=0.
\]
The general solution to this equation is given by
\[
G(t)=1 +\frac{t {}_2 F_1(1/3,2/3;5/3;t)}{(1-t)^{2/3}}+C\frac{t^{1/3}}{(1-t)^{2/3}},
\]
where $C$ is a constant.
For any value of $C$ the solution is equal to $1$ at $t=0$. To fit the boundary value at $t=1$ we can rewrite the hypergeometric function using Kummer's solutions. We have that
\begin{multline*}
{}_2 F_1(1/3,2/3;5/3;t)=\frac{\Gamma(-2/3)\Gamma(5/3)}{\Gamma(1/3)\Gamma(2/3)}(1-t)^{2/3}{}_2 F_1(4/3,1;5/3;1-t)\\+
\frac{\Gamma(2/3)\Gamma(5/3)}{\Gamma(4/3)\Gamma(1)}{}_2 F_1(1/3,2/3;1/3;1-t),
\end{multline*}
and this simplifies to
\[
-(1-t)^{2/3}{}_2 F_1(4/3,1;5/3;1-t)+
\frac{\Gamma(2/3)\Gamma(5/3)}{\Gamma(4/3)\Gamma(1)}t^{-2/3}.
\]
Plugging this expression into the general solution gives
\[
G(t)=1-t {}_2 F_1(4/3,1;5/3;1-t) + \frac{\Gamma(2/3)\Gamma(5/3)}{\Gamma(4/3)\Gamma(1)}\frac{t^{1/3}}{(1-t)^{2/3}}+C\frac{t^{1/3}}{(1-t)^{2/3}}
\]
which is equal to zero at $t=1$ if and only if
\[
C=-\frac{\Gamma(2/3)\Gamma(5/3)}{\Gamma(4/3)\Gamma(1)}
\]
and the last two terms in the expression for $G$ cancel out and we get \eqref{G}. 

With this definition of $G$ we have thus found a solution which we can write as
\[
\tilde{L}:=L(x,y)L(u,v)+\sqrt{L(x,y)}\sqrt{L(u,v)}\sqrt{1- L(x,y)}\sqrt{1-L(u,v)}G(\sigma).
\]

%\begin{equation}
%\label{G}
%G(\sigma) = 1-\sigma {}_2 F_1(4/3,1;5/3;1-\sigma).
%\end{equation}

%\begin{multline}
%\tilde{L}:=\frac{1}{4}\left(1+\frac{x}{\sqrt{x^2+y^2}}\right)\left(1+\frac{u}{\sqrt{u^2+v^2}}\right)\\
%\times \left(1+ \left(\frac{y}{x+\sqrt{x^2+y^2}}\frac{v}{u+\sqrt{u^2+v^2}}\right) \left[1-\sigma F(1-\sigma)\right]\right).  
%\end{multline}
It remains to show that $\tilde{L}(x,y,u,v)=L(x,y,u,v)$. (We do no
longer assume that $L(x,y,u,v)$ is $\mathcal{C}^2$.) To this end, we observe that It\^o's
formula implies that $\tilde{L}_t=\tilde{L}(x_t,y_t,u_t,v_t)$ is a
local martingale and it is clear that it is in fact a martingale. Since the path has zero probability of hitting any of the
two points it follows that $\lim_{t \to
  \infty}\tilde{L}_t \in \{0,1\}$  with probability
one. A calculation shows that the limit is equal to $1$ if and
only if both $x_t/y_t$ and $u_t/v_t$ tend to $+\infty$ as $t \to
\infty$. The probability of this event is by definition $L(x,y,u,v)$. Clearly $\tilde{L}(x,y,u,v)$ is bounded when ${|z| \to 0}$ (or $|w| \to 0$, or both). Consequently, $\tilde{L}_t$ is uniformly integrable and we get that
\[
\tilde{L}(x,y,u,v)=\mathbb{E}[\lim_{t \to \infty}\tilde{L}_t]=L(x,y,u,v)
\] 
and this concludes the proof.
\end{proof}
% By Koebe's one quarter theorem this also gives the asymptotic hitting probability of a small Euclidian ball up to %multiplicative constants. By \eqref{apr6.1} we have an asymptotic ``cutting probability'' for a small Euclidian %segment with the multiplicative constant.  
\section{Application to Bubbles and Multiple SLE }
We will now apply Theorem~\ref{mainthm} to derive several expressions for connectivity functions and quantities related to (hulls of) SLE$_{8/3}$ bubbles to be defined below. In particular, we will find an expression for the second moment of the area of the hull of an SLE bubble 
conditioned to have radius $1$ and we will use a version of the SLE$_{8/3}$ bubble to give a new construction of (the hull of) a pair of commuting SLE$_{8/3}$ paths. Combined with Theorem~\ref{mainthm}, this construction leads to a simple proof of a prediction by Gamsa and Cardy. 

\subsection{Bubbles}\label{bb}
Let us construct the SLE$_{8/3}$ bubble (pinned at $0$). We will use the M\"obius transformations 
\[
F_{\ee}(z)=\frac{z}{\ee-z}, \quad F_{\ee}^{-1}(z)=\frac{\ee z}{z+1},
\]
so that $F_{\ee}$ maps $\HH$ onto $\HH$ mapping $\ee$ to $\infty$ while fixing $0$.
Let $w \in \mathbb{H}$. We define the (hull of the) \textbf{SLE bubble with bulk point $w$} by considering the closure of the bounded component of the complement of the chordal SLE path from $0$ to $\ee >0$ conditioned to pass to the left of $w$ (see below) and then letting $\ee \to 0$. The existence of the limit follows, e.g., from arguments along the lines of \cite{SW10} Section 6.2; see also below. It is convenient to have a name for the closure of the bounded component of the complement of the chordal SLE path from $0$ to $\ee$ and we will call it an \textbf{SLE $\ee$-bubble} and write $P_{\ee}^{\textrm{SLE}}$ for its law. Thus, the SLE bubble with bulk point $w$ is the limit as $\ee \to 0$ of an $\ee$-bubble conditioned to contain $w$. The event that a point $w$ is contained in an $\ee$-bubble is formally defined as the event that the standard SLE path ``generating'' the bubble passes to the left of $F_{\ee}(w)$. We will also consider similar bubbles conditioned to have radius at least $0<R< \infty$, defined by conditioning the $\ee$-bubble on hitting the circle around the origin of radius $R$ and then letting $\ee \to 0$.  

Note that we work with probability measures on SLE bubbles; the SLE$_{8/3}$ and Brownian bubble measures considered in, e.g., Section 7 of \cite{LSW_restriction}, are $\sigma$ finite infinite measures on bubbles attached to the origin. (The mass of bubbles of radius at most $r$ blows up like $r^{-2}$ as $r \to 0$.) The SLE$_{8/3}$ bubble (non-probability) measure $\mu^{\textrm{SLE}}$ is constructed by renormalizing the law of an $\ee$-bubble by $\ee^{-2}$ and passing to the limit. Similarly, the Brownian bubble (non-probability) measure $\mu^{\textrm{B}}$ can be constructed by considering the filling of a Brownian excursion in $\mathbb{H}$ from $0$ to $\ee$ and renormalizing its law by $\ee^{-2}$. (The filling of a closed bounded set $S \subset \overline{\mathbb{H}}$ is the closure of the union of $S$ with the bounded components of $\mathbb{H} \setminus S$.) These two measures are very closely related as $\mu^{\textrm{SLE}} = (5/8) \mu^{\textrm{B}}$, see \cite{LSW_restriction}. We can construct probability measures from these measures by restricting to a set of strictly positive and finite measure and then renormalizing. For example, by similar computations as below, the $\mu^{\textrm{SLE}}$-measure of $E_R$, the set of bubbles of radius at least $0<R< \infty$, is strictly positive and finite. This gives an alternative construction of the probability measure on bubbles of radius at least $R$:
\[
\frac{\mu^{\textrm{SLE}}(\, \cdot \cap E_R)}{\mu^{\textrm{SLE}}(E_R)}=
\frac{\lim_{\ee \to 0}\ee^{-2}P_{\ee}^{\textrm{SLE}}(\, \cdot \cap E_R )}{\lim_{\ee \to 0}\ee^{-2}P_{\ee}^{\textrm{SLE}}(E_R) } = \lim_{\ee \to 0} \frac{P_{\ee}^{\textrm{SLE}}(\, \cdot \cap E_R)}{P_{\ee}^{\textrm{SLE}}(E_R)} =\lim_{\ee \to 0} P_{\ee}^{\textrm{SLE}}(\, \cdot \mid E_R). 
\]
Using that $\mu^{\textrm{SLE}} = (5/8) \mu^{\textrm{B}}$, we see that the conditioned SLE$_{8/3}$ probability measure also equivalently describes Brownian bubbles.

We will begin by computing probabilities of several events for $\ee$-bubbles. In what follows, $z$ and $w$ are two fixed points in $\mathbb{H}$, $R>\max\{|z|,\,|w|\}$, and $C_R=\partial D_R \cap \mathbb{H}$ where $D_R$ is the disk of radius $R$ centered at the origin. We also assume that $\ee$ is much smaller than all other parameters. Recall also that $\sigma(z,w)=|z-w|^2/|z- \overline{w}|^2$ and that the function $G$ was defined in \eqref{G}.

\begin{proposition}\label{formulas}
Let $R$, $z$, $w$, and $\epsilon$ be as above. Then the following statements hold.
\begin{enumerate}[(a)]
\item
The probability that $z$ is inside the SLE$_{8/3}$ $\ee$-bubble is given by
\begin{equation}
\label{Pz}
P(z)=\frac{1}{4}\left(\Im\left(\frac{1}{z}\right)\right)^2\ee^2+o(\ee^2).
\end{equation}
\item
The probability that the two points $z$ and $w$ are inside the SLE$_{8/3}$ $\ee$-bubble is given by
\begin{equation}
\label{Pzw}
P(z,w)=\frac{1}{4}\Im\left(\frac{1}{z}\right)\Im\left(\frac{1}{w}\right)G(\sigma(z,w))\ee^2+o(\ee^2).
\end{equation}

%\item
%The probability that the point $z$ is inside the SLE$_{8/3}$ $\ee$-bubble but $w$ is not is given by
%\begin{equation}
%\label{Pznotw}
%P^{lr}(z,w)=\frac{1}{4}\left(\Im \left(\frac{1}{z} \right) \right)^2 \ee^2- \frac{1}{4}\Im\left(\frac{1}{z}\right)\Im\left(\frac{1}{w}\right)G(\sigma(z,w))\ee^2+o(\ee^2).
%\end{equation}

\item
The probability that the SLE$_{8/3}$ $\ee$-bubble is inside $D_R$ and contains $z$ is given by
\begin{equation}
\label{PRz}
P_R(z)=\frac{\ee^2}{4 R^2}\left(\Im \left(J \left( \frac{z}{R} \right) \right)\right)^2+o(\ee^2),
\end{equation}
where $J(z)=z+z^{-1}$ is the Joukowsky map.
\item
The probability that the SLE$_{8/3}$ $\ee$-bubble is inside $D_R$ and contains the two points $z$ and $w$ is given by
\begin{equation}
\label{PRzw}
P_R(z,w)=\frac{\ee^2}{4R^2}\Im \left(J \left(\frac{z}{R} \right) \right)\Im \left(J \left(\frac{w}{R} \right) \right)G\left(\sigma \left(J \left(\frac{z}{R} \right),J \left(\frac{w}{R} \right)\right)\right)+o(\ee^2),
\end{equation}
where $J(z)=z+z^{-1}$ is the Joukowsky map.
%\item
%The probability that the SLE$_{8/3}$ $\ee$-bubble is inside $C_R$ and contains the point $z$ but not the point $w$ is given by
%\begin{multline}
%\label{PRznotw}
%P_R^{lr}(z,w)=\frac{\ee^2}{4R^2}\left(\Im\left(J\left( \frac{z}{R}\right)\right)\right)^2\\ - \frac{\ee^2}{4R^2}\Im \left(J \left(\frac{z}{R} \right) \right)\Im \left(J \left(\frac{w}{R} \right) \right)G\left(\sigma \left(J \left(\frac{z}{R} \right),J \left(\frac{w}{R} \right)\right)\right)+o(\ee^2),
%\end{multline}
%where $J(z)=z+z^{-1}$ is the Joukowsky map.
%
\end{enumerate}

\end{proposition}

\begin{proof}
The proofs of all these statements are straight-forward Taylor series computations involving the one- and two-point functions which are now convenient to write in complex form:
\begin{align}
\label{Lcomplex}
L(z)&=\cos^2(\arg(z)/2),
\\
\label{LLcomplex}
L(z,w)&=\cos^2(\arg(z)/2)\cos^2(\arg(w)/2) + \frac{1}{4}\sin(\arg(z))\sin(\arg(w))G(\sigma(z,w)).
\end{align}
(Note that this differs from our previous notation.)
Let us start with the first formula. By conformal invariance of SLE,
$P(z)=L(F_\epsilon(z))$. For small $\ee>0$
\[
F_\ee(z)=-1-\frac{\ee}{z}+O(\ee^2)
\]
and
\[
\arg(F_\ee(z))=\pi+\ee\Im(z^{-1})+O(\ee^2)
\]
plugging this into \eqref{Lcomplex} we get \eqref{Pz}.

Similarly,
\[
P(z,w)=L(F_{\ee}(z), F_{\ee}(w)) =  \ee^2\frac{1}{4}\Im(z^{-1})\Im(w^{-1}) G\left(\sigma(F_{\ee}(z), F_{\ee}(w)) \right)+O(\ee^4). 
\]
Note that by M\"obius invariance
\[
\sigma(F_{\ee}(z), F_{\ee}(w)) = \sigma(z,w),
\]
and this completes the proof of \eqref{Pzw}.

To prove the last two formulas we have to introduce more notation. The image of $C_R$ under $F_\ee$ is the semi-circle of radius $\rho$ centered at $-a$, where
\[
\rho=\frac{\ee R}{R^2-\ee^2}
\] 
and
\[
a=\frac{R^2}{R^2-\ee^2}.
\]
Let us denote the corresponding half-disc by $\ball$. 
By the restriction property, the law of a standard chordal SLE$_{8/3}$ curve conditioned to avoid $\ball$ is the same as the law of an SLE$_{8/3}$ curve in the complement with respect to $\mathbb{H}$ of $\ball$ (defined by conformal invariance). Moreover, the probability that a standard chordal SLE$_{8/3}$ curve avoids $\ball$ equals $|\psi'(0)|^{5/8}$, where
$\psi$ is the conformal map from $\HH\setminus \ball$ onto $\HH$ which preserves the origin and infinity and has derivative $1$ at infinity, see \cite{LSW_restriction}. It follows that $P_R(z)=|\psi'(0)|^{5/8}L(\psi(F_\ee(z)))$ and $P_R(z,w)=|\psi'(0)|^{5/8}L(\psi(F_\ee(z)),\psi(F_\ee(w)))$. It is easy to see that $\rho\to 0$ as $\ee\to0$ and that $|\psi'(0)|\to 1$, which means that this factor does not affect the leading term in the series for $P_R(z)$ and $P_R(z,w)$.

There is an alternative way to construct $\psi(F_\ee(z))$: it can be written (up to an irrelevant scaling) as $-1/J_R(z)$, where
\[
J_R(z)=-\frac{z}{R}-\frac{R}{z}+\frac{\ee}{R}+\frac{R}{\ee}=\frac{R}{\ee}\left(1-\frac{\ee}{R}\left(\frac{z}{R}+\frac{R}{z}\right)+\frac{\ee^2}{R^2}\right)
\]
is a conformal map from the half disc onto the upper half-plane which maps the origin to infinity and $\ee$ to the origin. 

We see that 
\[
\arg(J_R(z))=\frac{\ee}{R}\Im(J_R(z))+o(\ee)=\frac{\ee}{R}\Im(J(z/R))+o(\ee),
\]
where $J(z)=z+z^{-1}$ is the standard Joukowsky map.
The argument of $-1/J_R(z)$ (which is the same as the argument of $\psi(F_\ee(z))$) is 
\begin{equation}
\label{Jarg}
\pi-\frac{\ee}{R}\Im(J(z/R))+o(\ee).
\end{equation}
Plugging this into \eqref{Lcomplex} gives
\[
P_R(z)=\frac{\ee^2}{4R^2}\left(\Im(J(z/R))\right)^2+o(\ee^2).
\]
Finally, to derive \eqref{PRzw} we recall that $\sigma$ is a conformal invariant and 
\[
\sigma(\psi(F_\ee(z)),\psi(F_\ee(w)))=\sigma(J(z/R),J(w/R))
\]
which together with \eqref{Jarg} gives 
\[
\begin{aligned}
P_R(z,w)&=L(\psi(F_\ee(z)),\psi(F_\ee(w)))(1+o(\ee))
\\
&=\frac{\ee^2}{4R^2}\Im(J(z/R))\Im(J(w/R))G(\sigma(J(z/R),J(w/R)))+o(\ee^2).
\end{aligned}
\]
%The remaining expressions are proved in the same way using that the probability that the SLE$_{8/3}$ path passes to the left of
%$z$ and to the right of $w$ equals
%\[
%\frac{1}{4}((1+\cos(\arg(z)))(1-\cos(\arg(w)))) - \frac{1}{4}\sin(\arg(z))\sin(\arg(w))G(\sigma(z,w)).
%\]
This completes the proof.
\end{proof}

We can now derive several corollaries of the above formulas.

\begin{corollary}\label{area1}
For $w \in \mathbb{H}$ fixed, let $p_w(z)$ be the probability that $z$ is contained in the SLE$_{8/3}$ bubble with bulk point $w$. It holds that
\begin{equation}\label{apr8.4}
p_w(z)=\frac{\Im(z^{-1})}{ \Im(w^{-1})} G(\sigma(z,w)).
\end{equation}
Moreover, the expected area of the SLE$_{8/3}$ bubble with bulk point $w$ is infinite.
\end{corollary}

\begin{proof}
The first part is an immediate consequence of formulas \eqref{Pz} and \eqref{Pzw}. Indeed, for $\ee>0$ the probability that $z$ is inside the $\ee$-bubble conditioned to contain $w$ is given by
\[
\frac{P(z,w)}{P(w)}=\frac{\Im(z^{-1})}{ \Im(w^{-1})} G(\sigma(z,w))+o(\ee)
\]
which in the limit gives \eqref{apr8.4}. By Fubini's theorem, the expected area is given by
 \[
 \int_{\mathbb{H}}p_w(z) \, dA(z)=-\frac{1}{\Im(w^{-1})}\iint_{\mathbb{R}_+ \times [0,\pi]} \sin(\theta) G(\sigma) dr d\theta
 \]
 and this integral diverges since $G(\sigma(re^{i\theta},w)) \asymp r^{-1}$ if $z=re^{i\theta}$ is contained in a wedge. This proves the second claim.
\end{proof}

\begin{corollary}
Let $\mathcal{R}_z$ be the radius of the SLE$_{8/3}$ bubble with bulk point $z$, that is, $\mathcal{R}_z$ is the radius of the smallest disc centered at the origin containing the bubble. The distribution function of $\mathcal{R}_z$ is given by
\begin{equation}
\label{Rdistr}
\mathbb{P} \left\{ \mathcal{R}_z \le r \right\}=\left(1-\frac{|z|^2}{r^2}\right)^2, \quad r \ge |z|.
\end{equation} 
Consequently,
\[
\mathbb{E}[\mathcal{R}_z]=\frac{8}{3}|z|.
\]
\end{corollary} 
\begin{proof}
For fixed $z$ the probability that $\mathcal{R}_z \le r$ is given by 
\[
\frac{P_r(z)}{P(z)}=\frac{\left(\Im(J(z/r))\right)^2}{r^2\left(\Im(z^{-1})\right)^2}+o(\ee).
\]
Let $z=x+iy$. Then  
\[
\Im\left( J\left(\frac{z}{r} \right) \right)=\frac{y(|z|^2-r^2)}{r|z|^2}
\]
and 
\[
\Im\left(\frac{1}{z}\right)=-\frac{y}{|z|^2}.
\]
Using these expressions and passing to the limit with $\ee$ we get  \eqref{Rdistr}.
The expected radius is given by
\[
\int_{|z|}^\infty 4\left(1-\frac{|z|^2}{r^2}\right)\frac{|z|^2}{r^3}r \, dr=\frac{8}{3}|z|.
\]
\end{proof}

\begin{corollary}
Let $z$ and $w$ be two points in $D_R$. The probability that $z$ is inside the SLE$_{8/3}$ bubble with bulk point $w$ which is conditioned to stay inside $D_R$ is given by
\[
p_{w,R}(z)=\frac{\Im(J(z/R))}{\Im(J(w/R))}G(\sigma(J(z/R),J(w/R))),
\]
where $J(z)=z+z^{-1}$.
\end{corollary}

\begin{proof}
By the same argument as before, the probability $p_{w,R}(z)$ is the limit of the ratio $P_R(z,w)/P_R(w)$ as $\ee\to 0$.
\end{proof}

\subsection{Area distribution}
We will now consider an SLE$_{8/3}$ bubble of macroscopic size but with finite area. Formally, it is defined by fixing positive small $\ee, \delta$ and considering an $\ee$-bubble conditioned to intersect the unit circle but not the concentric circle of radius $1+\delta$. We then let $\ee \to 0$ and $\delta \to 0$ in that order so that we obtain a bubble conditioned to have radius $1$; see \cite{GTF} for more details. Let $\mathcal{A}$ be the area of (the hull of) such a bubble. Obviously, $\mathcal{A} \le \pi/2$. By Fubini's theorem we have that
\[
\mathbb{E}[\mathcal{A}]=\int_{\mathbb{D}_+}f(z) \, dA(z),
\]
where $f(z)$ is the probability that the point $z$ is contained in the bubble and $\mathbb{D}_+ = \mathbb{D} \cap \mathbb{H}$. Garban and Trujillo Ferreras used this identity and the one-point function to show that 
\[\mathbb{E}[\mathcal{A}] = \frac{\pi}{10},\] see Lemma~4.1 of \cite{GTF}. Similarly, it holds that
\begin{equation}\label{monster}
\mathbb{E}[\mathcal{A}^2]=\iint_{\mathbb{D}_+ \times \mathbb{D}_+}f(z,w) \, dA(z) \, dA(w),
\end{equation}
where $f(z,w)$ is the probability that $z$ and $w$ are simultaneously contained in the bubble. In this section we will use the two-point function to determine the probability $f(z,w)$. 

As briefly discussed in the introduction, there is some reason to expect, or at least investigate, whether (a constant times) the random variable $\mathcal{A}$ follows the Airy distribution. By properties of the Airy distribution, this would imply the following scale invariant identity:
\[
\frac{\mathbb{E}[\mathcal{A}^2]}{(\mathbb{E}[\mathcal{A}])^2} = \frac{10 }{3 \pi},
\]   
so that one would expect the integral \eqref{monster} to equal $\pi/30$. Since $f(z,w)$ is an explicit function this can of course in principle be checked, but it turns out that it is a rather complicated function that seems difficult to integrate exactly. (We comment on numerical evaluation below.)  
\begin{proposition}\label{apr13.3}
Let $z,w \in \mathbb{D}_+$ and let $f(z,w)$ be the probability that $z=x+iy$ and $w=u+iv$ are simultaneously contained in the hull of the SLE$_{8/3}$ bubble conditioned to touch the unit circle. Then,
\[
f(z,w)=\frac{2y v}{5|z|^2|w|^2}\left(|z|^2+|w|^2-2|z|^2|w|^2-\frac{A}{1-2(xu+yv)+|z|^2|w|^2}\right),
\]
where
\begin{multline*}
A= 
2\sigma(1-|z|^2)(1-|w|^2)(xu-yv-|z|^2|w|^2)
{}_2 F_1(1, 4/3; 5/3; 1 - \sigma_0)  \\+
\sigma_0|z-w|^2(1-|z|^2|w|^2){}_2 F_1(4/3,2; 5/3; 1 - \sigma_0)
\end{multline*}
and
\[
\sigma=\sigma(z,w),\quad\sigma_0=\sigma(J(z),J(w)).
\]
\end{proposition}
\begin{proof}
Recall the notation from the previous subsection. We have to compute the double limit
\[
\lim_{\delta\to 0}\lim_{\ee\to 0}\frac{P_{R+\delta}(z,w)-P_R(z,w)}{P_{R+\delta}-P_R},
\]
where 
\begin{equation}\label{aug4.1}
P_{R+\delta}=1-\frac{5}{8}\frac{\ee^2}{R^2}+\frac{5}{4}\frac{\ee^2\delta}{R^3}+O(\ee^4)
\end{equation}
%\[
%P_r=(1-\frac{\ee^2}{r^2})^{5/8}=1-\frac{5}{8}\frac{\ee^2}{r^2}+O(\ee^4)
%\]
is the probability that an $\ee$-bubble stays inside the disc of radius $R+\delta, \, \delta \ge 0$.

By \eqref{PRzw} we can compute the inner limit and get
\[
\lim_{\delta\to 0}\frac{\tilde P_{R+\delta}(z,w)-\tilde P_R(z,w)}{5\delta/4R^3}=\frac{4R^3\partial_R \tilde P_R(z,w)}{5}
\]
where 
\[
\tilde P_R(z,w)=\Im \left(J \left(\frac{z}{R} \right) \right)\Im \left(J \left(\frac{w}{R} \right) \right)G\left(\sigma \left(J \left(\frac{z}{R} \right),J \left(\frac{w}{R} \right)\right)\right)/4R^2
\] 
is the non-trivial part of \eqref{PRzw}.

The only thing left is to compute the partial derivative of $\tilde P_R(z,w)$ which is a very long but straight-forward computation. We are interested in the value for $R=1$ which slightly simplifies the computations. This gives the stated formula and completes the proof.
%The result is that
%\begin{multline}
%\[
%f(z,w)=\frac{2yv}{5|z|^2|w|^2}\left(|z|^2+|w|^2-2|z|^2|w|^2-\frac{A}{1-2(xu+yv)+|z|^2|w|^2}\right)
%\]
%\end{multline}
%where
%\begin{multline*}
%A=2\sigma(1-|z|^2)(1-|w|^2)(xu-yv-|z|^2|w|^2){}_2 F_1(1, 4/3; 5/3; 1 - \sigma_0))+\\
%\sigma_0|z-w|^2(1-|z|^2|w|^2){}_2 F_1(4/3,2; 5/3; 1 - \sigma_0)
%\end{multline*}
%and
%\[
%\sigma=\sigma(z,w),\quad\sigma_0=\sigma(J(z),J(w)).
%\]
%This completes the proof.
\end{proof}
Despite some effort, we have not been able to compute the integral \eqref{monster}. Numerical evaluation, however, gives a result which is within $2-3$ percent of the predicted value of $\pi/30$. (We used several of Mathematica's integration schemes, including the Adaptive Monte Carlo method.) This gives some weak support for the hypothesis that $\mathcal{A}$ follows the Airy distribution. But, since we do not know how to estimate the error in the numerical integration the result of course has to be taken with a large grain of salt.
 
We end the section with the following remarks. It is known (see \cite{GTF}) that $2\mathbb{E}[\mathcal{A}]=\mathbb{E}[\mathcal{A}_{\mathbb{H}}]$, where $\mathcal{A}_{\mathbb{H}}$ is the area of the hull of a half-plane Brownian bridge of time-length $1$. (This process can be constructed by letting the $x$-coordinate evolve according to a one-dimensional Brownian bridge of time-length $1$ and the $y$-coordinate according to an independent one-dimensional Brownian excursion of time-length $1$.) This in turn gives the expected value of the area of the hull of a two-dimensional Brownian loop of time-length $1$, given by considering $Z_t-tZ_1$ where $Z$ is a planar Brownian motion. The equivalence no longer holds for the second moments, but arguing as in the proof of Lemma~3.1 of \cite{GTF} we can see that 
\[\mathbb{E}\left[\left( \frac{\mathcal{A}_{\mathbb{H}}}{\mathcal{R}}\right)^2\right]=2\mathbb{E}[\mathcal{A}^2], \] 
where $\mathcal{R}$ is the random radius of the half-plane Brownian bridge. 

\subsection{Bubbles and multiple SLE$_{8/3}$}
It turns out that very large bubbles are close to a certain restriction measure and as a consequence of this we can obtain a simple proof (and generalization) of a prediction by Gamsa and Cardy \cite{GC}. In what follows, we take the limits in the sense of convergence in law with respect to the Hausdorff metric on closed subsets of $\overline{\mathbb{H}} \cup \{ \infty \}$. 
%One can proceed in several different ways, and we will consider both the limit as $y \to \infty$ of an SLE$_{8/3}$ bubble with bulk point $iy$ and the limit as $R \to \infty$ of an SLE$_{8/3}$ bubble conditioned to have radius at least $R$. We take the limits in the sense of convergence in law with respect to the Hausdorff metric on closed subsets of $\overline{\mathbb{H}} \cup \{ \infty \}$. It is not hard to identify the limit.
\begin{proposition}\label{apr13.1}
Let $\mu_y$, $y>0$, be the law of the hull of an SLE$_{8/3}$ bubble in $\mathbb{H}$ with bulk point $iy$. As $y \to \infty$ the measures $\mu_y$ converge weakly with respect to Hausdorff distance to the chordal restriction measure with restriction exponent $2$. 

Similarly, the law $\nu_R$ of the hull of an SLE$_{8/3}$ bubble in $\mathbb{H}$ conditioned to have radius at least $R$ converges in the same sense, as $R \to \infty$, to the chordal restriction measure with restriction exponent $2$.
\end{proposition}
\begin{proof}
We start with the first assertion. Subsequential weak limits exist by compactness. Let $\mu_{\infty}$ be any subsequential weak limit of the sequence of measures $\mu_y$. We first show that any sample $K$ from $\mu_{\infty}$ of satisfies
\begin{equation} \label{p2}
\mu_{\infty}(K \cap A = \emptyset) = \phi'_A(0)^2,
\end{equation}
where $A$ is a smooth hull bounded away from $0$ and $\phi_A: \mathbb{H} \setminus A \to \mathbb{H}$ is the conformal map such that $\phi_A(0)=0, \, \phi_A(z)/z \to 1$ as $z \to \infty$. Indeed, let $A$ be given as above and let $\mathcal{L}(z)$ be the event that the standard chordal SLE$_{8/3}$ path $\gamma$ passes to the left of $z$. (Assume also that $\ee>0$ is sufficiently small so that $A$ is bounded away from $[0,\ee]$.) By the restriction property, the probability that SLE$_{8/3}$ from $0$ to $\ee$ conditioned to pass to the left of $iy$ avoids $A$ equals
\begin{align*}
\frac{\mathbb{P}\, \{\mathcal{L}(F_\ee(iy)), \gamma \cap F_\ee(A) = \emptyset\}}{L(F_\ee(iy))} & = 
\frac{\mathbb{P} \, \{\mathcal{L}(F_\ee(iy))\mid \gamma \cap F_\ee(A) = \emptyset\} \mathbb{P} \, \{\gamma \cap F_\ee(A) = \emptyset\}}{L(F_\ee(iy))} \\
&= \phi'_{A_\ee}(0)^{5/8}\frac{L(\phi_{A_\ee}(F_\ee(iy)))}{\ee^2 y^{-2}/4 + O(\ee^4y^{-4}) },
\end{align*}
where $A_{\ee} = F_{\ee}(A)$. We have that
\[
\phi_{A_{\ee}}(z) = C_{\ee} \frac{z}{\phi_A(\ee)-z} \circ \phi_A \circ F^{-1}_\ee,
\]
where $C_{\ee} = \phi_A(\ee)/(\ee \phi_A'(\ee))$. Consequently, $\phi_{A_{\ee}}'(0) \to 1$ as $\ee \to 0$. We can also see that  
\begin{align*}
\lim_{y \to \infty}\lim_{\ee \to 0}\frac{L(\phi_{A_\ee}(F_\ee(iy)))}{\ee^2 y^{-2}/4} & = \lim_{y \to \infty} \lim_{\ee \to 0}\frac{\phi_A(\ee)^2\Im(-\phi_A(iy)^{-1})^2/4}{\ee^2 y^{-2}/4}\\
& = \phi_A'(0)^2,
\end{align*}
and we have obtained \eqref{p2}.

It is also easy to check that $\mu_{\infty}$ is supported on sets $K$ such that: $K$ is unbounded, $\overline{K} \cap \mathbb{R} = 0$, and $\mathbb{C} \setminus \overline{K}$ is connected. These properties together with \eqref{p2} imply that $\mu_{\infty}$ is the unique chordal restriction measure with restriction exponent $2$ ($\mathbb{P}_{2}$), see Proposition~3.3 of \cite{LSW_restriction}.

For the second assertion, write $\nu_{\infty}$ for any subsequential weak limit of the measures $\nu_R$; such weak limits again exist by compactness. We have that 
\[
\nu_{\infty}(K \cap A = \emptyset)=\lim_{R \to \infty} \lim_{\ee \to 0}\phi_{A_\ee}'(0)^{5/8}\frac{\mathbb{P}\, \{\gamma \cap \phi_{A_{\ee}}(F_{\ee}(C_R)) \neq \emptyset\}}{\mathbb{P}\, \{\gamma \cap F_{\ee}(C_R) \neq \emptyset\}},
\]
where $C_R$ is the circle of radius $R$ around the origin. By the computations in Proposition~\ref{formulas}, \eqref{aug4.1}, and the normalization of $\phi_A$ at infinity it follows that this limit equals $\phi_A'(0)^2$.
\end{proof}
Recall that $\mathbb{P}_{2}$ is the (two sided) restriction measure with restriction exponent $2$.
Let us make a few remarks. 
\vskip 12pt
\begin{itemize}\itemsep12pt

\item{It follows directly from the additive property of restriction exponents that samples from the limiting measure $\mathbb{P}_2$ are the same as those from the law of the hull of two independent Brownian half-plane excursions, see \cite{LSW_restriction}.}
\item{A sample from $\mathbb{P}_2$ can also be interpreted as the hull of a system of two \emph{standard} chordal SLE$_{8/3}$ paths conditioned to not intersect or equivalently two commuting standard chordal SLE$_{8/3}$ paths, see \cite{WW04} and \cite{LSW_SAW}.}    
\item{In view of Proposition~\ref{apr13.1} it is natural to interpret $\mathbb{P}_2$ as an infinite SLE$_{8/3}$ bubble, ``pinned'' at $0$ and $\infty$. By mapping $\infty$ to some $x \neq 0$ and fixing $0$ this is a two-pinned SLE bubble in the language of \cite{SW10}.}
\item{A fourth interpretation of $\mathbb{P}_2$ is that of the hull of a Brownian half-plane excursion conditioned to not have cut-points, see \cite{WW04} p.24 for a further discussion. Vir\'ag has conjectured that a suitable limit of the law of a so-called \emph{Brownian Bead} converges to $\mathbb{P}_2$, see \cite{Virag}. (Very roughly speaking, a Brownian excursion can be viewed as a ``necklace'' of Brownian beads joined at the cutpoints.) By the equivalence of (hulls of) SLE$_{8/3}$ and Brownian bubbles, Proposition~\ref{apr13.1} can be viewed as a verification of the pinned Brownian bubble version of this statement, at least on the level of hulls. (However, it is not clear whether there is a direct relation between pinned Brownian bubbles and Brownian beads.) }
\item{We mention finally that $\mathbb{P}_2$ has been predicted to be the scaling limit of the half-plane infinite self-avoiding polygon, see \cite{LSW_SAW}.}
\end{itemize}
\vskip 12pt
%On a heuristic level, this is natural from our construction. By the domain Markov property we could expect the limit to be given by first sampling a chordal SLE$_{8/3}$ from 0 to $\infty$ and then, conditional on the first path, sampling one from $\infty$ to $0$ in the domain on the right of the first path. This construction is asymmetric but 

From Proposition~\ref{apr13.1} and the work in Section~\ref{bb} we now get the following result which in part was predicted by Gamsa and Cardy. It is interesting to note that Gamsa and Cardy found the one-point function \eqref{CG} by deriving (and solving) a \emph{third}-order PDE by ``fusing'' the two SLE paths. In particular, the equation is not a so-called level two degeneracy equation that would easily translate to SLE language by arguing as in the proof of Theorem~\ref{mainthm}. 

Given two curves $\gamma$ and $\gamma'$ in $\mathbb{H}$ connecting $0$ and $\infty$, we define the hull of the system of these curves to be the closure of the union of the components of $\mathbb{H} \setminus (\gamma \cup \gamma')$ which do not have either of the positive or negative reals as part of their boundary. For example, if the curves only intersect at $0$ and $\infty$, then the hull is the closure of the region between the two curves.
\begin{proposition}\label{gamsa-cardy}
Let $z, w \in \mathbb{H}$. The probability that $z$ and $w$ are both contained in the hull of a system of two commuting SLE$_{8/3}$ paths equals
\begin{equation}\label{aug3.1}
-\frac{2}{5}\left(\Im (z) \, \Im ( w^{-1}) + \Im (z^{-1}) \, \Im (w) \right)G(\sigma(z,w)),
\end{equation}
where $G$ was defined in \eqref{G} and $\sigma=|z-w|^2/|z-\overline{w}|^2$.  
In particular, the probability that the point $z$ is contained in the hull of a system of two commuting SLE$_{8/3}$ paths equals
\begin{equation}\label{CG}
-\frac{4}{5}\Im(z) \Im (z^{-1}).
\end{equation}
\end{proposition}
\begin{remark}
As a consequence of the above result the probability that the point $z$ is in the hull of two commuting SLE$_{8/3}$ paths, but $w$ is not, equals
\[
-\frac{4}{5}\Im(z) \Im (z^{-1})+\frac{2}{5}\left(\Im (z) \, \Im ( w^{-1}) + \Im (z^{-1}) \, \Im (w) \right)G(\sigma(z,w)).
\]
Indeed, this is simply the difference between the expressions \eqref{CG} and \eqref{aug3.1}.
 
\end{remark}
\begin{proof}[Proof of Proposition~\ref{gamsa-cardy}]
We will keep the notation from Proposition~\ref{apr13.1} and Proposition~\ref{formulas}. Then in view of the second part of Proposition~\ref{apr13.1}, and the remarks following its proof, we find the desired probability by computing the limit
\begin{multline*}
\lim_{R \to \infty}\lim_{\ee \to 0} \frac{P(z,w)-P_R(z,w)}{\mathbb{P}\, \{ \gamma \cap F_\ee (C_R) \neq \emptyset\}}\\
=\lim_{R \to \infty} \frac{\Im(z^{-1}) \Im (w^{-1})G(\sigma(z,w)) - R^{-2}\Im(J(z/R)) \Im(J(w/R)) G(\sigma(J(z/R), J(w/R)))}{5R^{-2}/2},
\end{multline*}
where we used the expressions from Proposition~\ref{formulas}. Recalling that $J(z)=z+z^{-1}$ and $\sigma(z,w)=|z-w|^2/|z-\overline{w}|^2$, \eqref{aug3.1} follows.

For the one-point function, let us use the other construction of $\mathbb{P}_2$. The first part of Proposition~\ref{apr13.1} (and the remarks following its proof) shows that the one-point function is given by 
\[
\lim_{y \to \infty}p_{iy}(z)=-\frac{4}{5}\Im(z) \Im (z^{-1})=\frac{4}{5}\sin^2(\arg(z)),
\] 
and so \eqref{CG} follows directly from \eqref{apr8.4}.
\end{proof}
\begin{remark}
It follows from the one-point formula that the probability that a point $z$ is inside the hull of a system of two \emph{independent} SLE$_{8/3}$ curves equals $\sin^2(\arg(z))/2$. This restriction measure has restriction exponent $\alpha=5/4$. For $\alpha=5/8$, corresponding to a single SLE$_{8/3}$ curve, the probability that $z$ is on the path is zero. We may however note the following. Consider the probability $f_{\ee}(z)$ that the standard chordal SLE$_{8/3}$ path passes between the two points $z\pm \epsilon \eta$, where $z \in \mathbb{H}$ and $|\eta|=1$. Then from the two-point formula, which can be derived from Theorem~\ref{mainthm}, we readily get
\begin{equation}\label{apr6.1}
\lim_{\epsilon \to 0} \epsilon^{-2/3}f_{\epsilon}(z)=c_0 \, (\Im z)^{-2/3}\sin^2(\arg z),
\end{equation}
where $c_0=\Gamma(2/3)\Gamma(5/3)/(2 \Gamma(4/3))$; we recover the so-called Green function for SLE$_{8/3}$ in $\mathbb{H}$ with a multiplicative constant. (The Green function can be defined by the above limit with $f_{\epsilon}(z)$ replaced by the probability that the conformal radius of the complement of the (whole) curve in $\mathbb{H}$ with respect to $z$ is at most $2\epsilon$, see, e.g., \cite{LPC}.)
\end{remark}
\bibliography{sle} 
\bibliographystyle{abbrv}
\end{document}